\documentclass[12pt]{article}
\usepackage{amssymb,amsmath,amsthm,epsfig}
\usepackage[margin=25mm]{geometry}
\usepackage{float}
\usepackage{sectsty}
\usepackage{lipsum}
\usepackage[labelfont=bf]{caption}
\allsectionsfont{\centering\MakeUppercase}
\newtheorem{thm}{Theorem}
\newtheorem{lem}{Lemma}

\newtheorem{cor}{Corollary}

\newtheorem{prob}{Problem}
\newtheorem{rk}{Remark}

\renewenvironment{proof}{{\bfseries Proof.}}{}

\makeatletter
\renewcommand{\maketitle}{\bgroup\setlength{\parindent}{0pt}
\begin{flushleft}
  \textbf{\@title}

  \@author
\end{flushleft}\egroup
}
\makeatother

\begin{document}
\title{\bf On the augmented Zagreb index}

\author{AKBAR ALI$^{*}$, ZAHID RAZA \& AKHLAQ AHMAD BHATTI\\
Department of Mathematics, National University of Computer and Emerging Sciences, B-Block, Faisal Town, Lahore-Pakistan.\\
$^{*}$Corresponding author: akbarali.maths@gmail.com}

\date{}
\maketitle

\renewcommand{\abstractname}{ABSTRACT}
\begin{abstract}
Topological indices play an important role in Mathematical Chemistry especially in the quantitative structure-property relationship (QSPR) and quantitative structure-activity relationship (QSAR). Recent research indicates that the augmented Zagreb index ($AZI$) possess the best correlating ability among several topological indices. The main purpose of the current study is to establish some mathematical properties of this index, or more precisely, to report tight bounds for the $AZI$ of chemical bicyclic and chemical unicyclic graphs. A Nordhaus-Gaddum-type result for the $AZI$ (of connected graph whose complement is connected) is also derived.
\end{abstract}

\section*{Introduction}

All the graphs discussed in this article are simple, finite and undirected. For a graph $G$, the vertex set and edge set of $G$ will be denoted by $V(G)$ and $E(G)$ respectively. The degree of a vertex $u\in V(G)$ will be denoted by $d_{u}$, while the edge connecting the vertices $u$ and $v$ will be denoted by $uv$ (Harary, 1969). In chemical graphs, the vertices correspond to atoms while the edges represent covalent bonds between atoms (Karcher \& Devillers, 1990). The complement $\overline{G}$ of a graph $G$ has the vertex set $V(\overline{G})=V(G)$ and the edge set $E(\overline{G})=\{uv| uv\not\in E(G)\}$. The maximum and minimum vertex degree in a graph $G$ is denoted by $\Delta(G)$ and $\delta(G)$ (or simply by $\Delta$ and $\delta$) respectively. A vertex $u$ is pendant if $d_{u}=1$. By a chemical graph we mean a connected graph with $\Delta\leq4$. A graph $G$ is $r$-regular (or simply regular) if $d_{u}=r$ for every vertex $u$ of $G$. A graph is bicyclic (respectively unicyclic) if it has $n+1$ (respectively $n$ ) edges. Denote by $C_{n}$ and $P_{n}$ the cycle and path respectively on $n$ vertices respectively. Undefined notations and terminologies may be referred to (Harary, 1969; Bondy \& Murty 1976).

Topological indices are numerical parameters of a graph which are invariant under graph isomorphisms. Nowadays, there are many such indices that have found applications in Mathematical Chemistry especially in the quantitative structure-property relationship (QSPR) and quantitative structure-activity relationship (QSAR) (Devillers \& Balaban, 1999; Gutman \& Furtula, 2010; Todeschini \& Consonni, 2009). A large number of such indices depend only on vertex degree of the molecular graph. One of them is the atom-bond connectivity ($ABC$) index, proposed by Estrada \textit{et al.} (1998) and is defined as
\[ABC(G)=\sum_{uv\in E(G)}\sqrt{\frac{d_{u}+d_{v}-2}{d_{u}d_{v}}}.\]
This index provides a good model for the stability of linear and branched
alkanes as well as the strain energy of cycloalkanes (Estrada \textit{et al.}, 1998; Estrada, 2008). Details about this index can be found in (Furtula \textit{et al.}, 2012; Hosseini \textit{et al.}, 2014; Lin \textit{et al.}, 2013; Dimitrov, 2013, 2014; Chen \& Guo, 2012; Gutman \textit{et al.}, 2013; Imran \textit{et al.}, 2014; Das \textit{et al.}, 2012; Chen \textit{et al.}, 2012) and the references cited therein.

Inspired by work on the $ABC$ index, Furtula \textit{et al.} (2009) proposed the following modified version of the $ABC$ index and called it as augmented Zagreb index ($AZI$):
\[AZI(G)=\displaystyle\sum_{uv\in E(G)}\left(\frac{d_{u}d_{v}}{d_{u}+d_{v}-2}\right)^{3}.\]
The prediction power is better than the $ABC$ index in the study of heat of formation for heptanes and octanes (Furtula \textit{et al.} 2009). Gutman \& To\v{s}ovi\'{c} (2013) tested the correlation abilities of 20 vertex-degree-based topological indices for the case of standard heats of formation and normal boiling points of octane isomers, and they found that the $AZI$ yields the best results. Moreover, Furtula \textit{et al.} (2013) recently undertook a comparative study of the structure-sensitivity of twelve vertex-degree-based topological indices by using trees and concluded that the $AZI$ has the greatest structure sensitivity.

When a new topological index is introduced in Mathematical Chemistry, one of the important questions that need to be answered is for which members of a certain class of $n$-vertex (chemical or molecular) graphs this index assumes minimal and maximal values. Furtula \textit{et al.} (2009) answered this question for the class of all $n$-vertex acyclic chemical graphs. The motivation for the current study comes from this paper (Furtula \textit{et al.}, 2009), and especially the recently proven fact that the $AZI$ possess the best correlating ability among several topological indices has attracted our attention. The main purpose of the present article is to solve the following molecular problem:
\begin{prob} For a given set of unicyclic (bicyclic) molecules with fixed number of atoms, find those molecules which has minimum and maximum $AZI$ value.
\end{prob}

In (Furtula \textit{et al.}, 2009) the extremal properties of the $AZI$ of trees and chemical trees were studied. Huang \textit{et al.} (2012) and Wang \textit{et al.} (2012) gave sharp lower and upper bounds for the $AZI$ under various classes of connected graphs (e.g. trees, unicyclic graphs, bicyclic graphs, etc.) and characterized corresponding extremal graphs. Huang \& Bolian (2015) ordered graphs by the $AZI$ in several families of graphs (trees, unicyclic, bicyclic and connected graphs). Zhan \textit{et al.} (2015) characterized the $n$-vertex unicyclic graphs with the second minimal $AZI$ value and $n$-vertex bicyclic graphs with the minimal $AZI$ value.

In 1956, Nordhaus \& Gaddum (1956) gave tight bounds on the product and sum of the chromatic numbers of a graph and its complement. Since then, such type of results have been derived for several other graph invariants. Details about this topic can be found in the recent survey (Aouchiche \& Hansen, 2013) and the references cited therein.

Recently in (Ali \textit{et al.}, 2015), we established sharp lower and upper bounds for some topological indices of a certain family of chemical graphs known as polyomino chains. In this paper, firstly we consider the two families of chemical graphs (chemical bicyclic and chemical unicyclic graphs) and find those elements in each family for which the $AZI$ attains its minimum and maximum value by using the approach introduced by Furtula \textit{et al.} (2009). Secondly, we obtain a Nordhaus-Gaddum-type result for the $AZI$ of a connected graph whose complement is connected.

\section*{The AZI of chemical bicyclic and Chemical Unicyclic graphs}

In this section, we give tight bounds for the $AZI$ of chemical bicyclic and chemical unicyclic graphs. To do so, we need some lemmas and the following notations. Let $n_{i}(G)$ be the number of vertices of degree $i$ in a graph $G$, and $x_{i,j}(G)$ be the number of edges connecting the vertices of degree $i$ and $j$.

\begin{lem}\label{L1}
   If $B_{n}$ is a chemical bicyclic graph with $n$ vertices but has no pendent vertex, then
\[AZI(B_{n})<\frac{1376}{135}n+\frac{416}{15}.\]
\end{lem}

\begin{proof}
Note that $B_{n}$ is isomorphic to one of the graphs $\displaystyle B_{n}^{1},B_{n}^{2}$ shown in Fig. \ref{fig:1}.
\renewcommand{\figurename}{Fig.}
\begin{figure}[H]
    \centering
     \includegraphics[width=6in, height=1.5in]{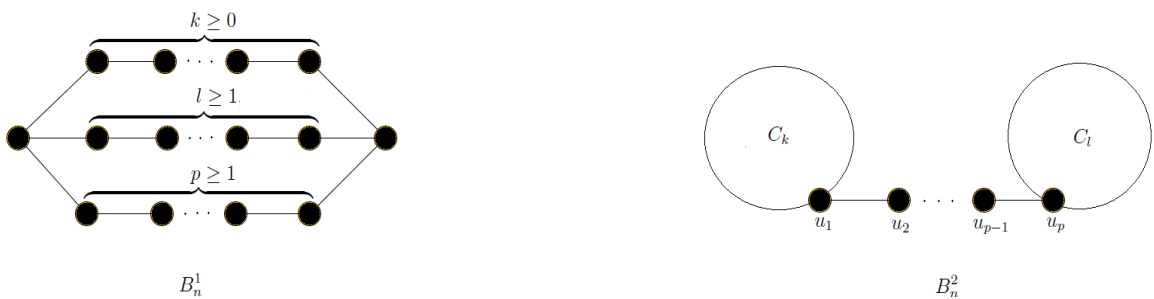}
     \caption{Two graphs $B_{n}^{1}$ and $B_{n}^{2}$, used in the proof of Lemma \ref{L1}}
      \label{fig:1}
\end{figure}
But,\\
$AZI(B_{n}^{1})=
\begin{cases}
       8n+\frac{729}{64} & \text{if    } k=0\\
       8(n+1) & \text{otherwise}
\end{cases}$
and
   $AZI(B_{n}^{2})=
\begin{cases}
       8n+\frac{729}{64} & \text{if    } p=2\\
       8(n+1) & \text{otherwise}
\end{cases}$
It can be easily seen that
\[AZI(B_{n}^{i})<\frac{1376}{135}n+\frac{416}{15}; \  i=1, 2.\]
\end{proof}

\begin{lem}\label{L2}
   Let $B_{n}$ be a chemical bicyclic graph with $n$ vertices such that
\begin{equation}\label{11111}
F(B_{n})\geq\frac{n+9}{5}\widetilde{\theta}_{4,4},
\end{equation}
where
\begin{equation}\label{2}
F(B_{n})=\sum_{uv\in E(B_{n})}\left(8-\left(\frac{d_{u}d_{v}}{d_{u}+d_{v}-2}\right)^{3}\right)=8(n+1)-AZI(B_{n}),
\end{equation}

and $\widetilde{\theta}_{i,j}= 8-\left(\frac{ij}{i+j-2}\right)^{3}=8-\theta_{i,j}$ . Then

\begin{equation}\label{1}
AZI(B_{n})\leq\frac{1376}{135}n+\frac{416}{15}
\end{equation}
\end{lem}

\begin{proof}
After substituting the value of $\widetilde{\theta}_{4,4}$ in (\ref{11111}), one have
\begin{equation}\label{3}
F(B_{n})\geq -\frac{296}{135}n-\frac{296}{15}.
\end{equation}
By using the identity (\ref{2}) in the above inequality (\ref{3}) and then after simple calculations, one obtained the desired result.

\end{proof}

\begin{thm}\label{t1}
   If $B_{n}$ is a chemical bicyclic graph with $n$ vertices, then
\begin{equation}\label{1}
AZI(B_{n})\leq\frac{1376}{135}n+\frac{416}{15}.
\end{equation}
If $B_{n}\cong B_{n}^{'}$ (where $B_{n}^{'}$ is depicted in Fig. \ref{fig:2}) then the bound is attained.
\end{thm}
\renewcommand{\figurename}{Fig.}
\begin{figure}[H]
    \centering
    \includegraphics[width=4.0in, height=1.5in]{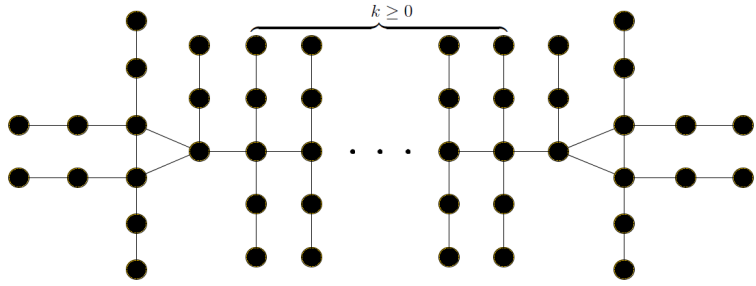}
    \caption{Chemical bicyclic graph $B_{n}^{'}$ where $n=5k+26$.}
    \label{fig:2}
\end{figure}
\begin{proof}
If $B_{n}$ has no pendent vertex, then from Lemma \ref{L1} the required result follows. So let us suppose that $B_{n}$ has at least one pendent vertex. We will prove the following inequality
\[F(B_{n})\geq\frac{n+9}{5}\widetilde{\theta}_{4,4},\]
where $F(B_{n})$ and $\widetilde{\theta}_{i,j}$ are defined in the Lemma \ref{L2}. The Inequality (\ref{1}) will, then, directly follow from Lemma \ref{L2}. Let us, contrarily, suppose that there exist some $n$ and some chemical bicyclic graph $B_{n}^{3}$ having $n$ vertices satisfying
\begin{equation}\label{4}
F(B_{n}^{3})<\frac{n+9}{5}\widetilde{\theta}_{4,4}.
\end{equation}
Among all such chemical bicyclic graphs, let $B_{n}^{4}$ be the one with minimum value of $x_{1,3}(B_{n}^{4})+x_{1,4}(B_{n}^{4})$.
We claim that $x_{1,3}(B_{n}^{4})=x_{1,4}(B_{n}^{4})=0$. Suppose to the contrary that $u,v\in V(B_{n}^{4})$ such that $d_{u}=1$ and $d_{v}=3$ or $4$. Consider the graph $B_{n+1}^{5}$ obtained from $B_{n}^{4}$ by subdividing the edge $uv$. Then
\[x_{1,3}(B_{n+1}^{5})+x_{1,4}(B_{n+1}^{5})<x_{1,3}(B_{n}^{4})+x_{1,4}(B_{n}^{4})\]
and
\begin{eqnarray*}
F(B_{n+1}^{5})&=&F(B_{n}^{4})+\widetilde{\theta}_{1,2}+\widetilde{\theta}_{2,d_{v}}-\widetilde{\theta}_{1,d_{v}}
<\frac{n+9}{5}\widetilde{\theta}_{4,4}-\displaystyle\min_{i=3,4}\widetilde{\theta}_{1,i}\\
&=&\frac{(n+1)+9}{5}\widetilde{\theta}_{4,4}-\frac{1}{5}\widetilde{\theta}_{4,4}-\widetilde{\theta}_{1,3}
<\frac{(n+1)+9}{5}\widetilde{\theta}_{4,4}.
\end{eqnarray*}
This contradicts the definition of $B_{n}^{4}$. Hence $x_{1,3}(B_{n}^{4})=x_{1,4}(B_{n}^{4})=0$. Now, consider the collection $\mathbb{B}_{1}$ of chemical bicyclic graphs $B_{n}^{6}$ satisfying:
\begin{description}
  \item[1)] $x_{1,3}(B_{n}^{6})=x_{1,4}(B_{n}^{6})=0,$
  \item[2)] $F(B_{n}^{6})<\frac{n+9}{5}\widetilde{\theta}_{4,4}.$
\end{description}

Since $B_{n}^{4}$ belongs to $\mathbb{B}_{1}$, this collection is non-empty. Condition 1) implies that $n_{2}(B_{n}^{6})\geq n_{1}(B_{n}^{6})$. Among all graphs in $\mathbb{B}_{1}$, let $B_{n}^{7}$ be one, having the smallest value of $n_{2}(B_{n}^{7})-n_{1}(B_{n}^{7})$. We claim that $n_{2}(B_{n}^{7})-n_{1}(B_{n}^{7})=0$. Contrarily, suppose that
$n_{2}(B_{n}^{7})-n_{1}(B_{n}^{7})>0$ and let $u,v,w\in V(B_{n}^{7})$ such that $uv,uw\in E(B_{n}^{7}), d_{u}=2$ and $d_{v},d_{w}\geq2$. Then we have three cases:\\
\\\textit{Case 1.} Exactly one of $d_{v},d_{w}$ is 2.\\
Without loss of generality we can assume that $d_{v}=2$ and $d_{w}\geq3$.

\textit{Subcase 1.1.} If $vw\in B_{n}^{7}$.
Let $B_{n+4}^{8}$ be the graph obtained from $B_{n}^{7}$ by adding the vertices $u_{1},u_{2},v_{1},v_{2}$ and edges $uu_{1},u_{1}u_{2},vv_{1},v_{1}v_{2}$. Then
\begin{eqnarray*}
F(B_{n+4}^{8})&=&F(B_{n}^{7})+2(\widetilde{\theta}_{3,d_{w}}-\widetilde{\theta}_{2,d_{w}})+(\widetilde{\theta}_{3,3}-\widetilde{\theta}_{2,2})+2\widetilde{\theta}_{3,2}+2\widetilde{\theta}_{1,2}\\
&<&\frac{n+9}{5}\widetilde{\theta}_{4,4}+2\displaystyle\max_{i=3,4}\widetilde{\theta}_{3,i}+\widetilde{\theta}_{3,3}
= \frac{(n+4)+9}{5}\widetilde{\theta}_{4,4}-\frac{4}{5}\widetilde{\theta}_{4,4}+3\widetilde{\theta}_{3,3}\\
&<&\frac{(n+4)+9}{5}\widetilde{\theta}_{4,4}.
\end{eqnarray*}
It means that $B_{n+4}^{8}\in \mathbb{B}_{1}$. But
\[n_{2}(B_{n+4}^{8})-n_{1}(B_{n+4}^{8})<n_{2}(B_{n}^{7})-n_{1}(B_{n}^{7}).\]
This is a contradiction to the definition of $B_{n}^{7}$.

\textit{Subcase 1.2.} If $vw\notin B_{n}^{7}$. Consider the graph $B_{n-1}^{9}$ obtained from $B_{n}^{7}$ by removing the vertex $u$ and adding the edge $vw$. Then
\[n_{2}(B_{n-1}^{9})-n_{1}(B_{n-1}^{9})<n_{2}(B_{n}^{7})-n_{1}(B_{n}^{7})\]
and
\begin{eqnarray*}
F(B_{n-1}^{9})&=&F(B_{n}^{7})+\widetilde{\theta}_{d_{v},d_{w}}-\widetilde{\theta}_{2,d_{w}}-\widetilde{\theta}_{2,d_{v}}=F(B_{n}^{7})\\
&<&\frac{n+9}{5}\widetilde{\theta}_{4,4}
<\frac{(n-1)+9}{5}\widetilde{\theta}_{4,4},
\end{eqnarray*}
which is again a contradiction to the definition of $B_{n}^{7}$.\\
\\\textit{Case 2.} $d_{v}=d_{w}=2$.\\
It is can be easily seen that $v$ and $w$ are non-adjacent. By using similar technique used in the Subcase 1.2, one obtains a contradiction.\\
\\\textit{Case 3.} $d_{v},d_{w}\geq3$.\\
Let $\widetilde{B}_{n+2}^{9}$ be the graph obtained from $B_{n}^{7}$ by adding the vertices $u_{1},u_{2}$ and edges $uu_{1},u_{1}u_{2}$. Then
\begin{eqnarray*}
F(\widetilde{B}_{n+2}^{9})&=&F(B_{n}^{7})+(\widetilde{\theta}_{3,d_{v}}-\widetilde{\theta}_{2,d_{v}})+(\widetilde{\theta}_{3,d_{w}}-\widetilde{\theta}_{2,d_{w}})+\widetilde{\theta}_{3,2}+\widetilde{\theta}_{1,2}\\
&<&\frac{n+9}{5}\widetilde{\theta}_{4,4}+\widetilde{\theta}_{3,d_{v}}+\widetilde{\theta}_{3,d_{w}}<\frac{n+9}{5}\widetilde{\theta}_{4,4}+2\displaystyle\max_{i=3,4}\widetilde{\theta}_{3,i}\\
&=&\frac{(n+2)+9}{5}\widetilde{\theta}_{4,4}+(2\widetilde{\theta}_{3,3}-\frac{2}{5}\widetilde{\theta}_{4,4})<\frac{(n+2)+9}{5}\widetilde{\theta}_{4,4}
\end{eqnarray*}
moreover,
\[n_{2}(\widetilde{B}_{n+2}^{9})-n_{1}(\widetilde{B}_{n+2}^{9})<n_{2}(B_{n}^{7})-n_{1}(B_{n}^{7})\]
which is a contradiction, again. In all three cases, contradiction is obtained. Hence $n_{2}(B_{n}^{7})-n_{1}(B_{n}^{7})=0$.

Now, let $\mathbb{B}_{2}$ be a sub collection of $\mathbb{B}_{1}$, consisting of those graphs $B_{n}^{10}\in\mathbb{B}_{1}$ which satisfy the property $n_{2}(B_{n}^{10})-n_{1}(B_{n}^{10})=0$. Note that the collection $\mathbb{B}_{2}$ is non-empty because $B_{n}^{7}\in\mathbb{B}_{2}$. (Let us denote by $n_{3}^{'}(G)$ the number of vertices of degree 3 adjacent to at least two vertices of degree greater than 2 in a graph $G$) Suppose that $B_{n}^{11}$ be a member of $\mathbb{B}_{2}$ having minimum value of $n_{3}^{'}(B_{n}^{11})$. We claim that $n_{3}^{'}(B_{n}^{11})=0$. On contrary, suppose that there exist a vertex $u$ of degree 3 adjacent to two vertices $v,w$ of degrees greater than 2 and to a vertex $z$ of degree greater than 1. Consider the graph $B_{n+2}^{12}$ obtained from $B_{n}^{11}$ by adding the vertices $u_{1},u_{2}$ and edges $uu_{1},u_{1}u_{2}$. Then $n_{3}^{'}(B_{n+2}^{12})<n_{3}^{'}(B_{n}^{11})$.
But
\begin{eqnarray*}
F(B_{n+2}^{12})&=&F(B_{n}^{11})+(\widetilde{\theta}_{4,d_{v}}-\widetilde{\theta}_{3,d_{v}})+(\widetilde{\theta}_{4,d_{w}}-\widetilde{\theta}_{3,d_{w}})+(\widetilde{\theta}_{4,d_{z}}-\widetilde{\theta}_{3,d_{z}})+\widetilde{\theta}_{4,2}+\widetilde{\theta}_{1,2}\\
&<&\frac{n+9}{5}\widetilde{\theta}_{4,4}+2\displaystyle\max_{i=3,4}(\widetilde{\theta}_{4,i}-\widetilde{\theta}_{3,i})+\displaystyle\max_{i=2,3,4}(\widetilde{\theta}_{4,i}-\widetilde{\theta}_{3,i})\\
&=&\frac{(n+2)+9}{5}\widetilde{\theta}_{4,4}-\frac{2}{5}\widetilde{\theta}_{4,4}+2\displaystyle\max_{i=3,4}(\widetilde{\theta}_{4,i}-\widetilde{\theta}_{3,i})<\frac{(n+2)+9}{5}\widetilde{\theta}_{4,4}.
\end{eqnarray*}
This contradicts the definition of $B_{n}^{11}$. Hence $n_{3}^{'}(B_{n}^{11})=0$.

From the facts $x_{1,3}(B_{n}^{11})=x_{1,4}(B_{n}^{11})=0$ and $n_{1}(B_{n}^{11})=n_{2}(B_{n}^{11})$, we deduce that no vertex of degree 2 lies on any cycle of $B_{n}^{11}$ which implies that no vertex of degree 3 lies on any cycle of $B_{n}^{11}$ because $n_{3}^{'}(B_{n}^{11})=0$. Note that in the graph $B_{n}^{11}$, each vertex of degree 3 is adjacent to two vertices of degree 2 that are adjacent to vertices of degree 1 which implies that $x_{2,2}(B_{n}^{11})=x_{3,3}(B_{n}^{11})=0$ and $n_{4}(B_{n}^{11})>0$. Hence, the vertices of degree 4 form a sub-bicyclic (connected) graph of $B_{n}^{11}$ and therefore
\begin{equation}\label{5}
x_{4,4}(B_{n}^{11})=n_{4}(B_{n}^{11})+1.
\end{equation}
Moreover, $x_{1,3}(B_{n}^{11})=x_{3,3}(B_{n}^{11})=n_{3}^{'}(B_{n}^{11})=0$ implies that
\begin{equation}\label{6}
x_{3,4}(B_{n}^{11})=n_{3}(B_{n}^{11}).
\end{equation}
Since $x_{1,3}(B_{n}^{11})=x_{1,4}(B_{n}^{11})=x_{3,3}(B_{n}^{11})=\widetilde{\theta}_{2,j}=0$ for all $j$, we have
\begin{equation}\label{7}
F(B_{n}^{11})=x_{3,4}(B_{n}^{11}).\widetilde{\theta}_{3,4}+x_{4,4}(B_{n}^{11}).\widetilde{\theta}_{4,4}.
\end{equation}
Using (\ref{5}) and (\ref{6}) in (\ref{7}), we get
\begin{equation}\label{8}
F(B_{n}^{11})=n_{3}(B_{n}^{11}).\widetilde{\theta}_{3,4}+n_{4}(B_{n}^{11}).\widetilde{\theta}_{4,4}+\widetilde{\theta}_{4,4}.
\end{equation}
Now, the relation $\displaystyle\sum_{i=1}^{4}i.n_{i}(B_{n}^{11})=2\left(\displaystyle\sum_{i=1}^{4}n_{i}(B_{n}^{11})+1\right)$ gives us
\[n_{1}(B_{n}^{11})=n_{3}(B_{n}^{11})+2n_{4}(B_{n}^{11})-2.\]
Bearing this preceding identity in mind and using the fact $n_{1}(B_{n}^{11})=n_{2}(B_{n}^{11})$, the equation $n(B_{n}^{11})=\displaystyle\sum_{i=1}^{4}n_{i}(B_{n}^{11})$ can be transformed to
\begin{equation}\label{9}
n_{4}(B_{n}^{11})=\frac{1}{5}\left(n(B_{n}^{11})-3n_{3}(B_{n}^{11})+4\right).
\end{equation}
From (\ref{8}) and (\ref{9}), it follows that
\begin{eqnarray*}
F(B_{n}^{11})&=&\frac{n(B_{n}^{11})+9}{5}\widetilde{\theta}_{4,4}+\left(\widetilde{\theta}_{3,4}-\frac{3}{5}\widetilde{\theta}_{4,4}\right)n_{3}(B_{n}^{11})\\
&\geq&\frac{n(B_{n}^{11})+9}{5}\widetilde{\theta}_{4,4}
\end{eqnarray*}
which is a contradiction to the Inequality (\ref{4}).\\
To prove that the bound is attainable, let us calculate the $AZI$ of the graph $B_{n}^{'}$  (see Fig. \ref{fig:2}).
\[AZI(B_{n}^{'})=(4k+20)8+(k+7)\left(\frac{8}{3}\right)^{3}=\frac{1376}{135}n+\frac{416}{15}.\]
\end{proof}
Denote by $\Psi_{n,m,\Delta}$ the collection of all those connected graphs $G$ having $n$ vertices, $m$ edges and maximum degree $\Delta$ in which $d_{u}=\Delta$ and $d_{v}=1$ or 2 for each edge $uv\in E(G)$. Wang \textit{et al.} (2012) gave the best possible lower bound for the $AZI$ of connected graphs:
\begin{lem}\label{L3}
(Wang \textit{et al.}, 2012). Let $G$ be a connected graph of order $n\geq3$ with $m$ edges and maximum degree
$\Delta$, where $2\leq\Delta\leq n-1$. Then
\[AZI(G)\geq\left(\frac{\Delta}{\Delta-1}\right)^{3}\left(2n-m-\frac{2m}{\Delta}\right)+8\left(2m-2n+\frac{2m}{\Delta}\right)\]
with equality if and only if $G\cong P_{n}$ for $\Delta=2$, and $G\in\Psi_{n,m,\Delta}$ with $m\equiv0(mod\Delta)$ for $\Delta\geq3$.
\end{lem}
As a consequence of Lemma \ref{L3}, we have:
\begin{cor}\label{C1}
If $B_{n}$ be a chemical bicyclic graph with $n$ vertices, then
\[AZI(B_{n})\geq\frac{4}{27}(35n+111)\]
equality holds if and only if $B_{n}\in\Psi_{n,n+1,4}$ with $n\equiv3\pmod{4}$.
\end{cor}
\begin{proof}
From the definition of $B_{n}$, it follows that $m=n+1$ and $\Delta=3$ or $4$. Hence

\[\left(\frac{\Delta}{\Delta-1}\right)^{3}\left(2n-m-\frac{2m}{\Delta}\right)+8\left(2m-2n+\frac{2m}{\Delta}\right)
=\begin{cases}
       \frac{1}{24}(155n+377) & \text{if $\Delta=3$,}\\
       \frac{4}{27}(35n+111) & \text{if $\Delta=4$.}
\end{cases} \]
By simple calculations, one have
\[\frac{1}{24}(155n+377)>\frac{4}{27}(35n+111).\]
Therefore, from Lemma \ref{L3} the desired result follows.
\end{proof}

\begin{rk}
By using the technique, adopted in the proof of Theorem \ref{t1}, we obtained the same lower bound as given in Corollary \ref{C1}.
\end{rk}
Now, combining the Theorem \ref{t1} and Corollary \ref{C1}, one have:
\begin{thm}\label{thm5}
Let $B_{n}$ be a chemical bicyclic graph with $n$ vertices, then
\[\frac{4}{27}(35n+111)\leq AZI(B_{n})\leq\frac{1376}{135}n+\frac{416}{15},\]
left equality holds if and only if $G\in\Psi_{n,n+1,4}$ with $n\equiv3\pmod{4}$. Moreover,
if $B_{n}\cong B_{n}^{'}$ then the right equality holds.
\end{thm}
Now, let us derive lower and upper bounds for chemical unicyclic graphs with $n$ vertices. For the unicyclic graph $U_{n}^{'}$ depicted in the Fig. \ref{fig:3}, one have
$$AZI(U_{n}^{'})=(4k+12)8+(k+3)\left(\frac{8}{3}\right)^{3}=\frac{1376}{135}n.$$
\renewcommand{\figurename}{Fig.}
\begin{figure}[H]
    \centering
    \includegraphics[width=3.5in, height=1.5in]{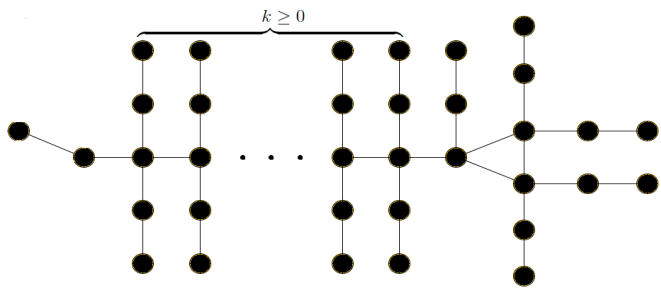}
    \caption{Chemical unicyclic graph $U_{n}^{'}$ where $n=5k+15$.}
    \label{fig:3}
\end{figure}
By using the same method as adopted to establish the Theorem \ref{thm5}, we have:
\begin{thm}
If $U_{n}$ is a chemical unicyclic graph with $n$ vertices, then
\[\frac{140}{27}n\leq AZI(U_{n})\leq\frac{1376}{135}n\]
left equality holds if and only if $G\in\Psi_{n,n,4}$ with $n\equiv0\pmod{4}$. Moreover, if $U_{n}\cong U_{n}^{'}$ then the right equality holds.
\end{thm}

\section*{Nordhaus-Gaddum-type results for AZI}

Nordhaus \& Gaddum (1956) gave tight bounds on the product and sum of the chromatic numbers of a graph and its complement. After their seminal work, such type of results have been derived for several other graph invariants, details can be found in the recent survey (Aouchiche \& Hansen, 2013). Here, we derive such kind of relation for the $AZI$. To proceed, we need some known results.
\begin{lem}
(Huang \textit{et al.}, 2012). Let $G$ be a connected graph with $m\geq2$ edges and maximum degree $\Delta$. Then
\begin{equation}\label{10}
AZI(G)\leq\frac{m\Delta^{6}}{8(\Delta-1)^{3}}
\end{equation}
with equality holding if and only if $G$ is a path or a $\Delta$-regular graph.
\end{lem}

A graph $G$ is said to be $(r_{1},r_{2})$-regular (or simply biregular) if $\Delta\neq\delta$ and $d_{u}=r_{1}$ or $r_{2}$, for every vertex $u$ of $G$.
Let $\Phi_{1}$ denote the collection of those connected graphs whose pendent edges are incident with the maximum
degree vertices and all other edges have at least one end-vertex of degree 2. Let $\Phi_{2}$ be the collection
of connected graphs having no pendent vertices but all the edges have at least one end-vertex of degree 2.
\begin{lem}
(Wang \textit{et al.}, 2012). Let $G$ be a connected graph of order $n\geq3$ with $m$ edges, $p$ pendent vertices,
maximum degree $\Delta$ and minimum non-pendent vertex degree $\delta_{1}$. Then
\begin{equation}\label{11}
AZI(G)\geq p\left(\frac{\Delta}{\Delta-1}\right)^{3}+(m-p)\left(\frac{\delta_{1}^{2}}{2\delta_{1}-2}\right)^{3}
\end{equation}
with equality if and only if G is isomorphic to a (1,$\Delta$)-biregular graph or G is isomorphic to a
regular graph or $G\in\Phi_{1}$ or $G\in\Phi_{2}$.
\end{lem}
Now, we are ready to prove the Nordhaus-Gaddum-type result for the $AZI$:
\begin{thm}
Let G be a connected graph of order $n\geq3$ such that its complement $\overline{G}$ is connected. Let $\Delta$, $\delta_{1}$, p and $\overline{\Delta}$, $\overline{\delta_{1}}$, $\overline{p}$ denote the maximum degree, minimal non-pendent vertex degree, the number of pendent vertices in G and $\overline{G}$ respectively. If $\alpha=\min\{\delta_{1},\overline{\delta_{1}}\ \}$ and $\beta=\max\{\Delta,\overline{\Delta}\ \}$, then
\begin{eqnarray}\nonumber\label{11a}
&&(p+\overline{p})\left(\frac{n-2}{n-3}\right)^{3}\left(1-\left(\frac{n-2}{2}\right)^{3}\right)
+\dbinom{n}{2}\left(\frac{\alpha^{2}}{2\alpha-2}\right)^{3}\\
&&\leq AZI(G)+AZI(\overline{G})\leq\dbinom{n}{2}\left(\frac{\beta^{2}}{2\beta-2}\right)^{3}
\end{eqnarray}
with equalities if and only if $G\cong P_{4}$ or $G$ is isomorphic to $r$-regular graph with $2r+1$ vertices.
\end{thm}
\begin{proof}
Suppose that $m$ and $\overline{m}$ are the number of edges in $G$ and $\overline{G}$ respectively. Firstly, we will prove the lower bound. Since both $G$ and $\overline{G}$ are connected, we have $\delta_{1}\leq\Delta\leq n-2$. Note that both the functions $f(x)=-\frac{x^{2}}{2x-2}$ and $g(x)=\frac{x}{x-1}$ are decreasing in the interval $[2,\infty)$, which implies that
\[-\frac{\displaystyle\delta_{1}^{2}}{2\delta_{1}-2}\geq-\frac{(n-2)^{2}}{2(n-3)}\ \text{and} \ \frac{\Delta}{\Delta-1}\geq\frac{n-2}{n-3}.\]
From (\ref{11}), we have
\begin{eqnarray}\nonumber\label{11b}
AZI(G)&\geq&p\left(\frac{n-2}{n-3}\right)^{3}+m\left(\frac{\delta_{1}^{2}}{2\delta_{1}-2}\right)^{3}
-p\left(\frac{(n-2)^{2}}{2(n-3)}\right)^{3}\\
&=&m\left(\frac{\delta_{1}^{2}}{2\delta_{1}-2}\right)^{3}
+p\left(\frac{n-2}{n-3}\right)^{3}\left(1-\left(\frac{n-2}{2}\right)^{3}\right)
\end{eqnarray}
this implies
\begin{eqnarray}\nonumber\label{12}
AZI(G)+AZI(\overline{G})&\geq&m\left(\frac{\delta_{1}^{2}}{2\delta_{1}-2}\right)^{3}
+\overline{m}\left(\frac{\overline{\delta_{1}}^{2}}{2\overline{\delta_{1}}-2}\right)^{3}\\
&+&(p+\overline{p})\left(\frac{n-2}{n-3}\right)^{3}\left(1-\left(\frac{n-2}{2}\right)^{3}\right).
\end{eqnarray}
Since the function $-f$ is increasing in the interval $[2,\infty)$ and $\delta_{1},\overline{\delta_{1}}\geq\alpha\geq2$, from (\ref{12}) it follows that
\begin{eqnarray}\nonumber\label{13}
AZI(G)+AZI(\overline{G})&\geq&m\left(\frac{\alpha^{2}}{2\alpha-2}\right)^{3}
+\overline{m}\left(\frac{\alpha^{2}}{2\alpha-2}\right)^{3}\\
&+&(p+\overline{p})\left(\frac{n-2}{n-3}\right)^{3}\left(1-\left(\frac{n-2}{2}\right)^{3}\right).
\end{eqnarray}
After using the fact $\overline{m}+m=\dbinom{n}{2}$ in (\ref{13}), one obtains the desired lower bound.
Now, we prove the upper bound. From (\ref{10}), it follows that
\begin{equation}\label{14}
AZI(G)+AZI(\overline{G})\leq\frac{m\Delta^{6}}{8(\Delta-1)^{3}}
+\frac{\overline{m}\displaystyle\overline{\Delta}^{6}}{8(\displaystyle\overline{\Delta}-1)^{3}}
\end{equation}
Since the function $h(x)=\frac{x^{6}}{8(x-1)^{3}}$ is increasing in the interval $[2,\infty)$ and $\Delta, \overline{\Delta}\geq2$, from Inequality (\ref{14}) we have
\begin{equation}\label{15}
AZI(G)+AZI(\overline{G})\leq\frac{m\beta^{6}}{8(\beta-1)^{3}}
+\frac{\overline{m}\displaystyle\beta^{6}}{8(\displaystyle\beta-1)^{3}}=\dbinom{n}{2}\left(\frac{\beta^{2}}{2\beta-2}\right)^{3}.
\end{equation}
Now, let us discuss the equality cases. If $G\cong P_{4}$ then $\overline{G}\cong P_{4}$ and if $G$ is isomorphic to $r$-regular graph with $2r+1$ vertices then $\overline{G}$ is also isomorphic to $r$-regular graph. Hence in either case, both lower and upper bounds are attained. Conversely, first let us suppose that left equality in (\ref{11a}) holds. Then all the Inequalities (\ref{11b}), (\ref{12}), (\ref{13}) must be Equalities.

\begin{description}
  \item[a)] Equality in (\ref{13}) implies that $\delta_{1}=\overline{\delta_{1}}$.
  \item[b)] Equality in (\ref{12}) implies\\
             \textit{(i).} $G$ is isomorphic to regular graph or $G\cong P_{4}$, and\\
             \textit{(ii).} $\overline{G}$ is isomorphic to regular graph or $\overline{G}\cong P_{4}$.
  \item[c)] Equality in (\ref{11b}) implies that either $\delta_{1}=\Delta=n-2$ and $p\neq0$ or
              $G$ is isomorphic to a regular graph.
\end{description}
Using the fact $P_{4}\cong \overline{P_{4}}$ and combining all the results derived in a), b), c), we obtain the desired conclusion. Finally, suppose that right equality in (\ref{11a}) holds, then both the Inequalities (\ref{14}), (\ref{15}) must be Equalities. Equality in (\ref{15}) implies that $\Delta=\overline{\Delta}=\beta$. Equality in (\ref{14}) implies that \\
\textit{i)} $G\cong P_{4}$ or $G$ is isomorphic to regular graph and \\
\textit{ii)} $\overline{G}\cong P_{4}$ or $\overline{G}$ is isomorphic to regular graph.\\
 Therefore, either $G\cong P_{4}$ or $G$ is isomorphic to $r$-regular graph with $2r+1$ vertices.
\end{proof}

\end{document}